\documentclass{article}
\usepackage[utf8]{inputenc}

\title{Tur\'an and Ramsey-type results for unavoidable subgraphs}
\author{Alp M\"uyesser\thanks{Freie Universit\"at, Institut für Mathematik and Berlin Mathematical School, email: \texttt{alp.muyesser@fu-berlin.de}.} \and Michael Tait\thanks{Villanova Univesity Department of Mathematics and Statistics, email: \texttt{michael.tait@villanova.edu}. Research is partially supported by National Science Foundation grant DMS-2011553.}}
\date{\vspace{-5ex}}
\usepackage{amsmath}
\usepackage{amssymb}
\usepackage{amsthm}
\usepackage{amsfonts}
\usepackage{mathtools}
\usepackage{ bbold }
\usepackage{changepage}
\usepackage{enumerate}
\usepackage{ dsfont }
\usepackage{soul}
\usepackage{hyperref}
\usepackage{tikz,tkz-graph}

\usepackage{listings}
\usepackage{xcolor}

\definecolor{codegreen}{rgb}{0,0.6,0}
\definecolor{codegray}{rgb}{0.5,0.5,0.5}
\definecolor{codepurple}{rgb}{0.58,0,0.82}
\definecolor{backcolour}{rgb}{0.95,0.95,0.92}

\lstdefinestyle{mystyle}{
    backgroundcolor=\color{backcolour},   
    commentstyle=\color{codegreen},
    keywordstyle=\color{magenta},
    numberstyle=\tiny\color{codegray},
    stringstyle=\color{codepurple},
    basicstyle=\ttfamily\footnotesize,
    breakatwhitespace=false,         
    breaklines=true,                 
    captionpos=b,                    
    keepspaces=true,                 
    numbers=left,                    
    numbersep=5pt,                  
    showspaces=false,                
    showstringspaces=false,
    showtabs=false,                  
    tabsize=2
}

\usepackage{braket}
\usepackage[a4paper, total={6in, 8in}]{geometry}

\newcommand{\ep}{\varepsilon}
\newcommand{\expected}{\mathbb{E}}

\newcommand{\bicolored}{bicolored }

\theoremstyle{plain}
\newtheorem{theorem}{Theorem}[section]
\newtheorem{lemma}[theorem]{Lemma}
\newtheorem{proposition}[theorem]{Proposition}
\newtheorem{claim}{Claim}[section]
\newtheorem{conjecture}{Conjecture}
\newtheorem{observation}{Observation}
\newtheorem{corollary}[theorem]{Corollary}
\newtheorem{definition}[theorem]{Definition}

\begin{document}
\maketitle

\begin{abstract}
    We study Tur\'an and Ramsey-type problems on edge-colored graphs. An edge-colored graph is called {\em $\ep$-balanced} if each color class contains at least an $\ep$-proportion of its edges. Given a family $\mathcal{F}$ of edge-colored graphs, the Ramsey function $R(\ep, \mathcal{F})$ is the smallest $n$ for which any $\ep$-balanced $K_n$ must contain a copy of an $F\in\mathcal{F}$, and the Tur\'an function $\textup{ex}(\ep, n, \mathcal{F})$ is the maximum number of edges in an $n$-vertex $\ep$-balanced graph which avoids all of $\mathcal{F}$. In this paper, we consider this Tur\'an function for several classes of edge-colored graphs, we show that the Ramsey function is linear for bounded degree graphs, and we prove a theorem that gives a relationship between the two parameters.

\end{abstract}

\section{Introduction}
 The central parameter in (graph) Ramsey theory is the \textit{Ramsey number} $R(k)$, quantifying the smallest integer $n$ for which every $2$-edge-coloring of the complete graph $K_n$ contains a monochromatic clique on $k$ vertices. It is known that $2^{k/2+o(1)}\leq R(k)\leq 2^{2k+o(1)}$, but the constants that appear in the exponents have resisted improvements for decades. For an overview, we refer the reader to the survey by Conlon, Fox, and Sudakov \cite{CFS}. We will refer to a graph endowed with a red-blue edge-coloring as a {\em \bicolored graph}. 
 \par Recently, there has been interest in finding non-monochromatic patterns in two-colorings of $K_n$. Of course, to do so, one needs to assume that both color classes are sufficiently well-represented. Bollob\'as thus asked (see \cite{CM}) which non-monochromatic subgraphs can be guaranteed in a two-coloring of $K_n$ where both color classes have at least $\ep\cdot e(K_n)$ edges. A \bicolored graph $G$ where both color classes have at least $\ep \cdot e(G)$ edges will be henceforth called {\em $\ep$-balanced}. Call a \bicolored graph $K_{2t}$ an \textit{unavoidable} $t$-graph if one of the color classes is a clique of size $t$ (Type 1), or two disjoint cliques of size $t$ (Type 2). As the coloring of $K_n$ can itself be unavoidable, subgraphs of unavoidable graphs are the only non-monochromatic subgraphs we can hope to find. Bollob\'as conjectured \cite{CM} that any large enough $\ep$-balanced complete graph would contain an unavoidable $t$-graph.
 \par Cutler and Montagh \cite{CM} confirmed Bollob\'as's conjecture. For $\ep\in (0,\frac{1}{2}]$ and $t\in\mathbb{N}$, let $R(\ep, t)\in\mathbb{N}$ be the smallest integer such that, whenever $n\geq R(\ep, t)$, every $\ep$-balanced $K_n$ contains an unavoidable $t$-graph. Sharp bounds for $R(\ep, t)$ were later obtained by Fox and Sudakov \cite{FS}, who showed that $R(\ep, t)=(1/\ep)^{\Theta(t)}$. Caro, Hansberg, and Montejano \cite{AAC} proved Bollob\'as's conjecture with even weaker hypotheses, letting $\ep \to 0$. They additionally studied the minimum density of each color class in a \bicolored $K_n$ required for finding forbidden \bicolored graphs with a fixed number of red edges, giving the problem a Tur\'an-type flavor. Subsequently, similar problems were studied in \cite{colorfulpaths, DHV, GN} (see also \cite{recent}). Multi-color and infinite variants of the problem were studied in \cite{BLM}. In \cite{AAC}, the connection between this problem and zero-sum labelings of graphs was explored (see also \cite{zero1}).
 
 \par In this paper, we continue the trend from \cite{AAC} and consider an extremal version of Bollob\'as's problem. The celebrated theorem of Tur\'an states that any graph $G$ with more than $\left(1-\frac{1}{k-1}\right)\binom{v(G)}{2}$ edges necessarily contains a $k$-clique. By analogy with Bollob\'as's problem and with Tur\'an's theorem, we ask which non-monochromatic subgraphs can be guaranteed in a dense enough $\ep$-balanced graph. It is clear as before that subgraphs of unavoidable graphs are the only patterns that we can hope to find. Our first theorem says that in fact unavoidable graphs can be found in dense enough $\ep$-balanced graphs, and the required density is closely related to the corresponding Ramsey function.
 \begin{theorem}\label{thm:mainextremal}
  Let $\ep\in(0,1/2]$ and $t\in \mathbb{N}$. For any $0< \ep'<\ep$ there exists a constant $K=K(\ep, \ep')$ such that any $\ep$-balanced graph $G$ with $e(G)\geq \left(1-\frac{1}{K\cdot R(\ep',t)}\right)\binom{v(G)}{2}$ edges (and sufficiently many vertices) contains an unavoidable $t$-graph.
 \end{theorem}
 On the other hand, one can construct a \bicolored graph without an unavoidable $t$-graph with density $\left(1-\frac{1}{R(\ep,t)-1}\right)$ as follows: Start with a $\ep$-balanced complete graph on $R(\ep,t)-1$ vertices that contains no unavoidable $t$-graphs, and replace each vertex by equally-sized independent sets, and each edge by a complete bipartite graph (of the same color). This shows that the asymptotics of this extremal problem is closely related to the Ramsey function $R(\ep,t)$. We leave as an open problem to determine whether a closer relationship can be established by eliminating the dependence of $K$ on $\ep$ and $\ep'$. 
 \vspace{2mm}
 \par Next we consider forbidding graphs other than the complete graph, and we see that our extremal theorem has an interesting application in the Ramsey setting. We call a \bicolored graph \textit{inevitable} if for large enough $t$ it can be embedded in both a Type 1 $t$-graph and Type 2 $t$-graph. For example, graphs where one color class form a star which spans all the vertices, and cycles whose coloring alternates between two red edges followed by two blue edges are inevitable. (We characterize inevitable graphs in Proposition \ref{inevitable characterization}.) If $H$ is an inevitable graph, $R(\ep, H)$ denotes the least integer such that every $\ep$-balanced clique on at least $R(\ep,H)$ vertices contains $H$ as a subgraph. The function $R(\ep, H)$ was studied in \cite{BLM} for dense $H$; it was shown that in this case $R(\ep,H)$ grows exponentially with~$v(H)$. For sparse $H$, one would expect $R(\ep, H)$ grow much slower. In the context of monochromatic graphs, Chvatal--Rödl--Szemer\'edi--Trotter showed \cite{CRST} that bounded degree graphs have linear Ramsey functions. Using Theorem 1.1, we obtain the colorful analogue of their result:
 \begin{theorem}\label{thm:linearramsey}
 For any $\ep>0$ and $\Delta$ there exists a constant $C=C(\ep, \Delta)$ such that for any inevitable $H$ with maximum degree $\Delta$, $R(\ep, H)\leq C\cdot v(H)$.
 \end{theorem}
 Lastly, we consider an extremal analogue of the $R(\ep, H)$ function. For $\mathcal{F}$ a family of \bicolored graphs, we call $\mathcal{F}$ {\em color-consistent} if it is unchanged by a permutation of the colors.
 \begin{definition}
     Let $\mathcal{F}$ be a color-consistent family of \bicolored graphs. We define $\mathrm{ex}(\ep, n, \mathcal{F})$ as the maximum integer $M$ such there exists a $n$-vertex $\ep$-balanced graph $G$ with $e(G)\geq M$ such that $G$ avoids all $F\in\mathcal{F}$. If $H$ is a single \bicolored graph, $\mathrm{ex}(\ep, n, H)$ denotes $\mathrm{ex}(\ep, n, \{H\})$. Finally, we let $\mathrm{ex}(\ep, \mathcal{F}):=\lim_{n\to\infty}\mathrm{ex}(\ep,n, \mathcal{F})/\binom{n}{2}$.
 \end{definition}
We prove basic properties of this function in Section \ref{extremal section}, such as the existence of the limit $\mathrm{ex}(\ep, \mathcal{F})$. We generalize the celebrated Erd\H{o}s-Stone theorem to our setting, which allows us to extend the following result of DeVos, McDonald, and Montejano (rephrased in our language).
 \begin{theorem}[DeVos--McDonald--Montejano \cite{DMM}]
 Let $T_1$ and $T_2$ be the \bicolored triangles that are not monochromatic. Then, $\mathrm{ex}(\frac{1}{2}, \{T_1, T_2\})=2/3$.
 \end{theorem}
 Thus, a density of $1/3$ in both color classes is required to ensure the existence of a non-monochromatic triangle. We extend this result to cycles of any fixed size:
 \begin{theorem}\label{cycle prop}
   Let $\mathcal{C}_k$ denote the family of all non-monochromatic \bicolored $k$-cycles. Then, if $k\geq 4$, $\mathrm{ex}(\frac{1}{2}, \mathcal{C}_k)=1/2$.
 \end{theorem}
 For non-monochromatic cliques, we show that the Tur\'an density which guarantees an uncolored clique also guarantees a \bicolored clique of size differing by at most an additive constant.
 \begin{theorem}\label{clique prop}
   Let $k\geq 3$, and let $\mathcal{K}_k$ denote the family of all non-monochromatic \bicolored $k$-cliques. Then,  $\mathrm{ex}(\frac{1}{2}, \mathcal{K}_k)=1-\frac{1}{k+O(1)}$.
 \end{theorem}
 
We believe Theorem \ref{clique prop} can be improved as follows.

\begin{conjecture}\label{nonmono clique conjecture}
Let $\mathcal{K}_k$ be the family of all non-monochromatic \bicolored $k$-cliques. For $k$ sufficiently large, $\mathrm{ex}(\frac{1}{2},\mathcal{K}_k) = 1-\frac{1}{k-1}$.
\end{conjecture}
\subsection*{Organization}
In the next section, we first prove our main extremal result, Theorem \ref{thm:mainextremal}, and then we show its application in the Ramsey setting, Theorem \ref{thm:linearramsey}, in Subsection \ref{linear ramsey}. Afterwards, in Section \ref{extremal section} we start studying the function $\mathrm{ex}(\ep,\mathcal{F})$. After proving some basic properties, we characterize when $\mathrm{ex}(\ep,\mathcal{F})<1$, and prove a multi-color version of the Erd\H{o}s--Stone theorem. In Section \ref{nonmono section}, we use this theory to prove Theorems \ref{cycle prop} and \ref{clique prop}.  

For $f,g:\mathbb{N}\to \mathbb{N}$, we use $f\ll g$ to mean that $f=o(g)$ and $f\lesssim g$ to mean $f \leq (1+o(1))g$. For a \bicolored graph $G$, we will use subscripts $R$ and $B$ to denote graph invariants restricted to the red or blue edges respectively, e.g. $E_B(G)$ will denote the set of blue edges in $G$ and $d_R(v)$ will denote the red degree of a vertex $v$.
\section{Balanced cliques in balanced graphs}
The proof of Theorem \ref{thm:mainextremal} is almost immediate after the following lemma.
\begin{lemma}\label{lem:mainlemma}
For any $0<\ep'<\ep\leq 1/2$, there exist a $C, K_0$ such that for all $k\geq K_0$ and $n$ sufficiently large, we have that any $\ep$-balanced $n$-vertex graph with $(1-\frac{1}{k})\binom{n}{2}$ edges contains an $\ep'$-balanced clique on at least $\frac{k}{C}$ vertices.
\end{lemma}
We do not know if in the above statement $\frac{k}{C}$ can be replaced by a quantity of the form $(1-o(1))k$. We consider this problem in the Discussion section further.

\begin{proof}[Proof of Theorem \ref{thm:mainextremal}, assuming Lemma \ref{lem:mainlemma}]
Given $\ep$ and $\ep'$, let $C$ and $K_0$ be the constants obtained by applying Lemma \ref{lem:mainlemma} with $\ep$ and $\ep'$. Choose $K:=\max\{\frac{K_0}{R(\ep', t)}, C\}$ and consider a graph $G$ with $e(G)\geq \left(1-\frac{1}{K\cdot R(\ep',t)}\right)\binom{|V|}{2}$ and $v(G)$ sufficiently large. Since by our choice, $K\cdot R(\ep, t)\geq K_0$, we have that $G$ contains an $\ep'$-balanced clique on $\frac{K\cdot R(\ep, t)}{C} \geq R(\ep', t)$ vertices. Such a subgraph by definition contains an unavoidable $t$-graph, as desired.
\end{proof}
Thus, the problem is reduced to finding the smaller clique which is almost as balanced. We do this in the following proof. Prior to the proof, we record the version of Chebyschev's inequality and the Chernoff bound that we will use (see \cite{AS}).
\begin{lemma}\label{chebyshev}[Chebyschev's inequality] Let $A$ be a random variable with finite mean and variance. Then the following inequality holds for any real $\delta>0$:
$$\mathbb{P}(|A-\mathbb{E}[A]|\geq \delta \cdot\mathbb{E}[A]) \leq \frac{Var[A]}{\delta^2\mathbb{E}(A)^2}.$$
\end{lemma}

\begin{lemma}\label{chernoff}[Chernoff bound] Let $X:=\sum_{i=1}^m X_i$ where $(X_i)_{i\in[m]}$ is a sequence of independent indicator random variables with $\mathbb{P}(X_i=1)=p_i$. Let $\expected[X]=\mu$. Then, for any $0<\gamma<1$ $\mathbb{P}(|X-\mu|\geq \gamma \mu)\leq 2e^{-\mu \gamma^2/3}$.
\end{lemma}
\begin{proof}[Proof of Lemma \ref{lem:mainlemma}]
We will sample a subset of vertices randomly, alter it to be a clique, and show that with positive probability that it is balanced. We will choose $C$ to be a large constant and $K_0$ will be chosen after $C$ is fixed. First, we choose a subset $S$ of vertices of $G$ by picking each vertex independently with probability $p: = \frac{4k}{Cn}$. Here $n$ is assumed to be large enough that $p<1$.

For every pair of non-adjacent vertices in $S$, we remove both of them. We call the resulting set $T$. We will show that there is a positive probability that $T$ satisfies the conclusion of the statement, that is, $T$ is an $\ep'$-balanced clique on at least $\frac{k}{C}$ vertices.

The bad events which we hope to simultaneously avoid are as follows:
\begin{enumerate}[(a)]
    \item $|T|<\frac{k}{C}$, \label{1.1 size}
    \item $e_B(T)\leq \ep'\binom{|T|}{2}$, \label{1.1 blue}
    \item $e_R(T)\leq \ep'\binom{|T|}{2}$. \label{1.1 red}
\end{enumerate}

Since $T$ is a clique by construction, if we avoid conditions \eqref{1.1 size}, \eqref{1.1 blue}, and \eqref{1.1 red} then we are done. Since $G$ has total edge density $1-\frac{1}{k}$ the expected number of pairs of vertices in $S$ which are not adjacent is
\[
p^2 \frac{1}{k}\binom{n}{2}\leq \frac{8k}{C^2}.
\]
Since we will remove at most $2$ vertices for each pair of non-adjacent vertices in $S$, by Markov's inequality the probability that we remove more than $\frac{160k}{C^2}$ vertices from $S$ is at most $\frac{1}{10}$. We also have that $\mathbb{E}(|S|) = \frac{4k}{C}$ and so by the Chernoff bound we have that
\begin{equation}\label{S size}
\mathbb{P}\left( \frac{2k}{C} < |S| < \left(\frac{1}{1-\frac{1}{C}}\right) \frac{4k}{C}\right) > \frac{9}{10},
\end{equation}
as long as $k$ is a large enough constant (note again that we choose $K_0$ after $C$). For ease of notation, let $M=\left(\frac{1}{1-\frac{1}{C}}\right) \frac{4k}{C}$. Choosing $C$ large enough so that $\frac{8k}{C^2} < \frac{k}{C}$ gives that \eqref{1.1 size} occurs with probability less than $\frac{1}{5}$. 

Next we give an upper bound the probability that \eqref{1.1 blue} occurs. Let $U$ be the event that the upper bound in \eqref{S size} occurs, i.e. that $|S| < M$. Recall that $K_0$ will be chosen so that $k$ is large enough to ensure $\mathbb{P}(U) > \frac{9}{10}$. To estimate the probability that \eqref{1.1 blue} occurs, we see that
\begin{align*}
    \mathbb{P}\left(\mbox{\eqref{1.1 blue} occurs}\right) &= \mathbb{P}\left(\mbox{\eqref{1.1 blue} occurs} \wedge U\right) + \mathbb{P}\left(\mbox{\eqref{1.1 blue} occurs} \wedge U^C\right)\\
    & \leq \mathbb{P}\left( e_B(T) \leq \ep' \binom{M}{2}\right) + \mathbb{P}(U^C),
\end{align*}

where the last inequality follows because if $U$ occurs then $|T| \leq |S| \leq M$. Therefore, we have that 
\[
\mathbb{P}\left(\mbox{\eqref{1.1 blue} occurs}\right) \leq \mathbb{P}\left( e_B(T) \leq \ep' \binom{M}{2}\right) + \frac{1}{10},
\]
and so it suffices to give a good enough upper bound on 
\[
\mathbb{P}\left( e_B(T) \leq \ep' \binom{M}{2}\right).
\]

To do this, we first wish to give a lower bound on $\expected[e_B(T)]$.  Let $\overline{N}(u)$ denote the non-neighborhood of a vertex and $\overline{d_u} = |\overline{N}(u)|$. Suppose $uv$ is a blue edge and note:
\[
\mathbb{P}(uv\in T) = \mathbb{P}(uv\in S)\cdot \mathbb{P}(uv \in T|uv\in S) = p^2\cdot \prod_{w\in \overline{N}(u) \cup \overline{N}(v)}\left( 1 - \mathbb{P}(w\in S) \right)\geq p^2(1-p)^{\overline{d_u} + \overline{d_v}}.
\]
Therefore, we have that
\begin{align*}
    \expected[e_B(T)]\geq \sum_{uv\in E_B(G)} p^2\left(1-p\right)^{\overline{d_u}+\overline{d_v}}
    &\geq p^2\sum_{uv\in E_B(G)}\left(1-(\overline{d_u}+\overline{d_v})p\right)\\
&=p^2e_B(G) - p^3\sum_{v\in V(G)}\overline{d_v}\cdot d_{v,blue}\\
& \geq p^2e_B(G) - \frac{2p^3n}{k}\binom{n}{2} \\
& \geq p^2e_B(G)\left(1 - \frac{16}{\varepsilon C}\right),
    \end{align*}

where the inequalities are by linearity of expectation, Bernoulli's inequality, $d_{v, blue} \leq n$, and simplification respectively. It follows that 
\[
\frac{\mathbb{E}[e_B(T)]}{\binom{M}{2}} \geq \ep\left(1- \frac{1}{k}\right)\left(1 - \frac{16}{\ep C}\right)\left(1- \frac{1}{C}\right)^2.
\]
Thus, assuming $K_0 \geq C$, we may choose $C$ large enough with respect to $\ep$ and $\ep'$ so that
\[
\frac{\mathbb{E}[e_B(T)]}{\binom{M}{2}} (1-\delta) \geq \ep',
\]
where $\delta := 1/C^{1/2}$ is an extra factor that we will use soon when applying Chebvyshev's inequality.

 Now, we wish to estimate $Var[e_B(T)]$ in order to apply Chebyschev's inequality. Let $\mathbb{1}_{uv}$ denote the indicator random variable denoting whether the blue edge $uv$ is inside $S$. Then $e_B(S)=\sum_{uv\in E_B(G)}\mathbb{1}_{uv}$. Note that for disjoint edges $uv$ and $xy$, $\mathbb{1}_{uv}$ and $\mathbb{1}_{xy}$ are independent random variables.    \begin{align*}Var[e_B(T)]&=\expected[e_B(T)^2]-\expected[e_B(T)]^2\leq \expected[e_B(S)^2]-\expected[e_B(T)]^2\\
&\leq \left(\sum_{|uv\cap xy|=0}\expected[\mathbb{1}_{uv}]\expected[\mathbb{1}_{xy}] + \sum_{|uv\cap xy|=1}\expected[\mathbb{1}_{uv}\mathbb{1}_{xy}]  + \sum_{|uv\cap xy|=2}\expected[\mathbb{1}_{uv}\mathbb{1}_{xy}]\right)  - \expected[e_B(T)]^2\\
&\leq \left(e_B(G)^2p^4 + n^3p^3 + n^2p^2\right)  - \expected[e_B(T)]^2\\
&\leq \mathbb{E}[e_B(T)]^2,
\end{align*}
where the last inequality is true if we choose $K_0$ to be a large enough constant.

With all the inequalities we have collected, we can now bound the probability that bad event \eqref{1.1 blue} happens via Chebyschev:
\begin{align*}
    \mathbb{P}\left(\mbox{\eqref{1.1 blue} occurs}\right) &\leq \mathbb{P}\left(e_B(T) \leq \ep'\binom{M}{2}\right) + \frac{1}{10}\\
    &\leq \mathbb{P}\left( \left| e_B(T) - \mathbb{E}[e_B(T)]\right| \geq \delta \cdot \mathbb{E}[e_B(T)]\right) +\frac{1}{10} \\
    &\leq \frac{\delta^{-2}Var[e_B(T)]}{\mathbb{E}[e_B(T)]^2} +\frac{1}{10}\leq \delta^{-2} + \frac{1}{10} = \frac{1}{C} + \frac{1}{10} < \frac{1}{5}.
\end{align*}

By symmetry, the probability that \eqref{1.1 red} occurs is also less than $\frac{1}{5}$. Using the union bound, there is a positive probability that none of the events occur.
\end{proof}
\subsection{Balanced Ramsey numbers of sparse graphs}\label{linear ramsey}
In this section we show how to prove Theorem \ref{thm:linearramsey} using Theorem \ref{thm:mainextremal}. We begin by formally stating the analogous result in the ordinary Ramsey setting. Recall that when $H$ is a graph $R(H)$ denotes the smallest integer $N$ such that for all $n\geq N$ all two-colorings of $K_n$ contains a monochromatic copy of $H$.
\begin{theorem}[Chvatal-R\"odl-Szemer\'edi-Trotter]\label{chvatal}
For every $\Delta$ there exists a $c$ such that $R(H)\leq c\cdot v(H)$ for all graphs $H$ with maximum degree at most $\Delta$.
\end{theorem}
A stronger conjecture was made by Erd\H{o}s and Burr \cite{EB}, replacing the word ``degree" with ``degeneracy", which is now proved by Lee \cite{Lee}. We present a sketch of the original proof of the theorem, which relies on the Regularity lemma, as our proof of Theorem \ref{thm:linearramsey} will follow this closely. Regularity free proofs of this result were obtained that give significantly better bounds on $c=c(\Delta)$ by (see \cite{FS2}).

Before we begin with the sketch, we state the Regularity lemma, starting with the necessary terminology. Let $G:=(A,B)$ be a bipartite graph with $|A|=|B|=n$. For $X\subset A$ and $Y\subset B$ define $d(X,Y) = \frac{e(X,Y)}{|X||Y|}$. We call $G$ $\ep$-regular if for all subsets $X\subseteq A$ and $Y\subseteq B$ with $|X|, |Y|\geq \ep n$ we have $|d(X,Y)-d(A,B)|\leq \ep$. 
\begin{lemma}[Szemer\'edi \cite{regularity}]\label{regularity lemma}
For any $\ep>0$ there exists an $M:=M(\ep)$ such that any graph $G$ can be partitioned into $k$ (where $\frac{1}{\ep}\leq k \leq M$) equal sized parts $(G_i)_{i\in [k]}$ and a junk set $J$ with $|J|<\ep n$ such that all but $\ep$-fraction of the pairs $(G_i,G_j)$ are $\ep$-regular.  
\end{lemma}
We sometimes call the pairs of the form $(G_i, G_j)$ \textit{super-edges}, the $G_i$ \textit{super-vertices}. Given $\ep$ and $d$, the graph with vertex set the super-vertices and $G_i$ joined to $G_j$ when the pair $(G_i, G_j)$ is $\ep$-regular with density at least $d$ is called the \textit{cluster graph}. We can now proceed with the sketch of Theorem \ref{chvatal}.
\begin{proof}[Sketch of proof of Theorem \ref{chvatal} \cite{CRST}]
For a carefully chosen $\ep$, consider an $\ep$-regular partition of a sufficiently large \bicolored $K_n$. If a pair is $\ep$-regular in the red edges it is also $\ep$-regular in the blue edges. Our goal is to find a subgraph in the cluster graph with $\Delta + 1$ parts. Indeed, afterwards, using $\chi(H)\leq \Delta + 1$, as well as $\Delta(G)\leq \Delta$, and ensuring each cluster is sufficiently large, a simple greedy algorithm shows that $H$ can be embedded in the graph (we will repeat this part of the argument in more detail in the proof of Theorem \ref{thm:linearramsey}).
\par To achieve the goal, we first eliminate all edges between non $\ep$-regular pairs in the regularity partition, yielding a cluster graph with density at least $1-\ep$. by Tur\'an's theorem, we can find a clique in this graph of size $1/\ep$. Associate with each super-edge $(G_i, G_j)$ the majority color class in the complete bipartite graph between $G_i$ and $G_j$ (ties can be broken arbitrarily). We can now apply Ramsey's theorem with the colors of the super-edges to obtain a monochromatic clique in the cluster graph. As we control $\ep$, we can ensure that this clique is on at least $\Delta+1$ super-vertices, concluding the proof.
\end{proof}
Since our proof of Theorem \ref{thm:linearramsey} closely resembles the argument in \cite{CRST}, we put the details in the appendix. Here, we will point out only the key differences with the sketch of the earlier result. First, compared to the first paragraph of the earlier sketch, our goal will change to finding an unavoidable $(\Delta + 1)$-graph in the cluster graph. Secondly, to achieve this goal, instead of applying Tur\'an's theorem to the cluster graph, we will apply our Theorem \ref{thm:mainextremal}. And thirdly, we will need a colored version of the Regularity lemma.
By simply iterating the regularity lemma to the red edges, and then the blue edges, we can obtain the following corollary.
\begin{corollary}\label{regularity corollary}
For any \bicolored graph $G$, we can obtain a single regularity partition where the red and the blue graphs both satisfy the conclusions of the Regularity lemma, with the same exact hypotheses. 
\end{corollary}
A minor difference besides these three will be that we will
need to keep track of how balanced our graph is as we go through the regular clean-up procedure, as well as when we are coloring the super-edges (in particular, coloring by the majority color class will not work here, as this might result in a monochromatic coloring). To achieve this we state a lemma that allows us to extract a balanced set of red and blue super-edges from the cluster graph of a balanced graph. As we will need this lemma again, we provide a proof in the paper. 
\par We will use the version of the Chernoff bound from Lemma \ref{chernoff} in our proof. 
\begin{lemma}\label{lem:clusterlemma}
Let $\ep, \ep_0>0$, $k\geq 1/\ep_0$, and $G$ be a $k$-partite \bicolored graph with parts $\mathcal{R}:=\{P_1,\cdots, P_k\}$ such that $d_R(P_i,P_j)\geq \ep$ or $d_B(P_i,P_j)\geq \ep$ or $d(P_i,P_j)=0$ for any two parts $P_i$ and $P_j$. Then, there exists an edge-coloring of a graph on vertex set $\mathcal{R}$ such that a red (resp. blue) edge $\{P_i, P_j\}$ implies $d_R(P_i,P_j)\geq \ep$ (resp $d_B(.)$), and furthermore, $|d_R(\mathcal{R})-d_R(G)|\leq \gamma_R$, where $\gamma_R=4\ep_0/\sqrt{d_R(G)}$, and the same holds for blue.  
\end{lemma}
\begin{proof}
\par We produce the coloring randomly, and bound the failure probability via the Chernoff bound. For any pair in $\mathcal{R}$ with positive density of edges, we add a red edge $\{P_i, P_j\}$ with probability $d_R(P_i, P_j)/d(P_i, P_j)$, and we add a blue edge otherwise. 
\par Let $X_{i,j}$ be the indicator random variable for whether $\{P_i, P_j\}$ is red or not in $\mathcal{R}$. Let $X:=\sum_{\{i,j\}\in \binom{[k]}{2}}X_{i,j}$. Then,
$$\expected[X]=\sum_{\{i,j\}\in \binom{[k]}{2}}\mathbb{P}(X_{i,j}=1)=d_R(G)\binom{k}{2}$$.
\par By the Chernoff bound,
$$\mathbb{P}(|X-\expected[X]|\geq \gamma \expected[X])\leq 2e^{-\expected[X]\gamma^2/3}\leq 2e^{-d_R(G)\frac{16k^2}{6}\ep_0^2d_R(G)^{-1}}\leq 2e^{-2}\leq \frac{2}{7}$$
\par Where in the second to last inequality we used $k\geq \frac{1}{\ep_0}$. Repeating the same calculation with the blue edges, we conclude that with at least $3/7$ probability, there exists a colored edge assignment in which neither the red nor the blue density deviate more than $\gamma_R$ or $\gamma_B$ from their prior densities in $G$.
\end{proof}
\section{The balanced extremal function}\label{extremal section}

\par In this section we systematically study the function $\mathrm{ex}(\ep, n, \mathcal{F})$ defined in the introduction. First, recall that a \bicolored graph $G$ is called $\ep$-balanced if both color classes have at least $\ep\cdot e(G)$ edges, and that a family $\mathcal{F}$ of \bicolored graphs is color-consistent if it is unchanged by a permutation of the colors. As defined in the introduction, given a color-consistent family of \bicolored graphs $\mathcal{F}$, we let $\mathrm{ex}(\ep, n, \mathcal{F})$ be the maximum integer $M$ such that there exists an $n$-vertex $\ep$-balanced \bicolored graph with $M$ edges that avoids all $F\in \mathcal{F}$. We also defined $\mathrm{ex}(\ep, \mathcal{F}) = \lim_{n\to \infty} \mathrm{ex}(\ep, n, \mathcal{F}) / \binom{n}{2}$.

Several remarks are in order about the definitions. First, we will be concerned with the case when $\mathcal{F}$ is composed of connected, non-monochromatic graphs. Thus $\mathrm{ex}(\ep, n, \mathcal{F})$ will always be at least quadratic in $n$, as two disjoint cliques of the same size but different colors will not contain any forbidden subgraphs. Indeed, we have $\mathrm{ex}(\ep, n, \mathcal{F})\geq \frac{1}{4}\binom{n}{2}$. Hence we are simply interested with the limit value, $\mathrm{ex}(\ep, \mathcal{F})$. Of course, it is not a priori clear that this limit exists. We deal with this technicality in the next lemma.
\begin{proposition} \label{lem:limitexists}
For any family of graphs $\mathcal{F}$ and any $\ep>0$, $\mathrm{ex}(\ep, \mathcal{F}):=\lim_{n\to\infty}\mathrm{ex}(\ep, n, \mathcal{F})/\binom{n}{2}$
exists.
\end{proposition}
\begin{proof}
Assume to the contrary. Since $\mathrm{ex}(\ep, n, \mathcal{F})$ is bounded between $0$ and $1$, we may fix two subsequences which converge to different limits. Say the two limits are $2\delta$ apart, for some $\delta>0$. We can then fix $n_1$ from the sequence with the larger limit and $n_2$ from the other sequence so that $n_1>n_2$ and 
\[
\frac{\mathrm{ex}(\ep, n_1, \mathcal{F})}{\binom{n_1}{2}} > \frac{\mathrm{ex}(\ep, n_2, \mathcal{F})}{\binom{n_2}{2}} + \delta,
\]
We note that we may choose $n_2$ to be as large as we need. Let $G_1$ be an $\mathcal{F}$-avoiding graph on $n_1$ vertices and $\textup{ex}(\ep, n_1, \mathcal{F})$ edges. We will use this graph to construct an $n_2$-vertex $\ep$-balanced $\mathcal{F}$-avoiding graph with density strictly larger than $\frac{\mathrm{ex}(\ep, n_2, \mathcal{F})}{\binom{n_2}{2}}$, yielding a contradiction. 
\par We sample a subset of $n_2$ vertices from $G_1$ (without repetition) uniformly at random, and call it $G_2$. Let $p(n):=\prod_{0\leq i<n}\frac{n_2-i}{n_1-i}$. Note that $\mathbb{E}[e_B(G_2)]=e_B(G_1)p(2)$. 
\par We now estimate $Var[e_B(G_2)]$. Let $\mathbb{1}_{uv}$ denote the random variable indicating whether the blue edge $uv$ is contained in $G_2$ so that $e_B(G_2)=\sum_{uv\in E_B(G_1)}\mathbb{1}_{uv}$.  \begin{align*}&Var[e_B(G_2)]=\expected[e_B(G_2)^2]-\expected[e_B(G_2)]^2\\
&= \left(\sum_{\substack{uv, xy\in E_B(G_1)\\|\{u,v,x,y\}|=4}}\expected[\mathbb{1}_{uv}\mathbb{1}_{xy}] + \sum_{\substack{uv, xy\in E_B(G_1)\\|\{u,v,x,y\}|=3}}\expected[\mathbb{1}_{uv}\mathbb{1}_{xy}]  + \sum_{\substack{uv, xy\in E_B(G_1)\\|\{u,v,x,y\}|=2}}\expected[\mathbb{1}_{uv}\mathbb{1}_{xy}]\right)  - \expected[e_B(G_2)]^2\\
&\leq \left(e_B(G_1)^2p(4) + n_1^3p(3) + n_1^2p(2)\right)  - \expected[e_B(G_2)]^2 = o(\expected[e_B(G_2)]^2).
\end{align*}
\par Thus, since the variance is low, for any $\gamma>0$, if $n_2$ is sufficiently large, we have by Chebyschev's inequality (Lemma \ref{chebyshev}) that $$\mathbb{P}(|e_B(G_2)-\expected[e_B(G_2)]|\geq \gamma\cdot \expected[e_B(G_2)])<1/2.$$
It follows that we may choose $n_2$ large enough that with probability more than $1/2$, $|d_B(G_2)-d_B(G_1)|\leq \delta/5$
\par We can similarly show that with probability more than $1/2$, the red density of $G_2$ is at most $\delta/5$ away from that of $G_1$, and so with positive probability, both events happen simultaneously. For such a $G_2$, by deleting at most $\frac{2\delta}{5}\frac{\binom{n_2}{2}}{2}$ edges, we can create an $\ep$-balanced $G_2'$. Overall, we can thus ensure that $G_2'$ has density strictly larger than $\frac{\mathrm{ex}(\ep, n_2, \mathcal{F})}{\binom{n_2}{2}}$, is $\ep$-balanced, and is $\mathcal{F}$-avoiding (as a subgraph of $G_1$), which gives us the desired contradiction. 
\end{proof}
We finally remark that it is clear from the definition that $\mathrm{ex}(\ep, n, \mathcal{F})$ is monotone increasing for decreasing $\ep$.
\subsection{Inevitable graphs and characterization}
Now that we know $\mathrm{ex}(\ep, \mathcal{F})$ exists for every family $\mathcal{F}$, the first natural question is to determine for which graphs $\mathrm{ex}(\ep, \mathcal{F})<1$. The answer turns out to be exactly those families containing a \bicolored subgraph which can be embedded in both types of unavoidable graphs. We defined such graphs to be \textit{inevitable} in the introduction section. Here we give a structural characterization of such \bicolored graphs. 
\begin{proposition}\label{inevitable characterization}
Let $H$ be a \bicolored graph. Then, $H$ is inevitable if and only if there exists a vertex partition $L\sqcup R = V(H)$ with the edges contained in $L$ and $R$ entirely red, and edges that go across entirely blue, such that either:
\begin{enumerate}
    \item $H$ does not contain a blue-red-blue walk on $3$ edges.
    \item $H$ does not contain a red-blue-red path,
\end{enumerate}
or the same except permuting red and blue.

\end{proposition}
Note that the first case can be replaced equivalently with: There exists subsets $X\subseteq L$, $Y\subseteq R$ such that $X$ and $Y$ are independent and all the edges that are between $L$ and $R$ are contained between $X$ and $Y$. A related characterization for a slightly different problem was given in Theorem 2.4 of \cite{AAC}. Our proof of Proposition \ref{inevitable characterization} is similar to that in \cite{AAC} but set in our context.

\begin{proof}
\par Any $H$ that falls in one of the two cases can be embedded into both types of unavoidable graphs and thus is inevitable. Indeed, embedding such $H$ into a Type $2$ graph is trivial because of the first part of the proposition. To embed a Case $1$ $H$ into a Type $1$ graph we simply embed the independent sets inside the isolated color class. To embed a Case $2$ $H$ into a Type $1$ graph we can greedily embed all the (for example) red edges inside the isolated (red) clique, and as there cannot be any blue edge between two red edges, we will not have a problem embedding the blue edges as well. We now focus on the other direction of the proof. 
\par Assume that $H$ is inevitable. The first part of the proposition follows immediately from the fact that there exists an embedding of $H$ into a Type $2$ unavoidable graph, where (without loss of generality) the edges of the bipartite graph is blue. We also now $H$ can be embedded in a Type $1$ unavoidable graph. We case on the color of the bipartite graph in the Type $1$ graph $H$ can be embedded to. 
\par \textbf{Case 1: } $H$ has an embedding to a Type $1$ unavoidable graph where the edges that go across are red. Let's call the set of vertices embedded in the blue part $B$, and the red part $R$. Now, observe that if any two vertices from $L$ are embedded in $B$, they cannot be adjacent, as edges contained in $L$ are red, but edges contained in $B$ are blue. Same goes for any two vertices from $R$. Therefore, any blue edge between vertices from $L$ and $R$ are contained in independent sets. 
\par \textbf{Case 2: } $H$ has an embedding to a Type $1$ unavoidable graph where the edges that go across are blue. Now it is clear that any path that starts red-blue will end on a vertex adjacent only to blue edges. 
\end{proof}
\par We finally state the following theorem, the proof of which now follows from the definitions.
\begin{theorem} Let $\mathcal{F}$ be a family of \bicolored graphs. The following are equivalent.
\begin{enumerate}[(1)]
    \item $\mathrm{ex}(\ep, \mathcal{F})<1$
    \item $R(\ep, \mathcal{F})<\infty$
    \item There exists $F, G\in \mathcal{F}$ such that $F$ and $G$ are contained in a Type $1$ graph, and a Type $2$ graph, respectively. 
\end{enumerate}
\end{theorem}
It may be that $F=G$ in $(3)$, in which case $\mathcal{F}$ contains an inevitable graph. 

\subsection{Erd\H{o}s-Stone type theorem}

In this subsection, we prove a colorful Erd\H{o}s-Stone theorem using the parameter $\mathrm{ex}(\ep,\, . \,)$. The proof uses the Regularity lemma (Lemma \ref{regularity lemma} and Corollary \ref{regularity corollary}) and additionally the following embedding lemma.




\begin{lemma}\label{embedding lemma}
Let $H$ be some graph embeddable in a sufficiently large Tur\'an graph on $n$ parts. For any $d>0$ there exists an $\ep>0$ with the following property: Let $R:=\cup_{i\in[n]} R_i$ be a subset of a Tur\'an graph on $n$ parts where each part has size large enough and if ${v_i,v_j}$ is an edge in $H$, then the pair $(R_i,R_j)$ is $\ep$-regular with density at least $d$. Then, $H$ is embeddable in $R$.  
\end{lemma}
The proof and explicit bounds follow from a simple inductive argument. For example, see Lemma 7.3.2 in Diestel \cite{Diestel}. We use the embedding lemma with \bicolored graphs as follows:

\begin{lemma}\label{colored embedding lemma}
Let $H$ be a \bicolored graph whose uncolored version is embeddable in a sufficiently large Tur\'an graph on $n$ parts. For any $d>0$ there exists an $\ep>0$ with the following property: Let $R:=\cup_{i\in[n]} R_i$ be a colored graph whose uncolored version is a subset of a Tur\'an graph on $n$ parts where each part has size large enough and if ${v_i,v_j}$ is an edge of color $c$ in $H$, then the pair $(R_i,R_j)$ is $\ep$-regular with $c$-density at least $d$. Then, $H$ is embeddable in $R$.  
\end{lemma}

 We may now state and prove our multicolor generalization of the celebrated Erd\H{o}s-Stone theorem.
\begin{theorem}\label{ES theorem}
Let $\ep, \delta>0$ and $t\in \mathbb{N}$ be arbitrary. Let $\mathcal{F}$ be some finite family of \bicolored graphs. Let $c=\mathrm{ex}(\ep, \mathcal{F})$. Then, for all sufficiently large $n$ (with respect to $\mathcal{F}$, $\ep$, $t$ and $\delta$) we have that any $\ep$-balanced $n$-vertex graph with $(c+\delta)\binom{n}{2}$ edges contains a $t$-blow-up of some $F\in \mathcal{F}$.
\end{theorem}

\begin{proof}
We give a brief outline of the proof first. We take a graph satisfying our assumptions, and apply the colored Regularity lemma (Corollary \ref{lem:clusterlemma}) to it. Then we create an auxiliary cluster graph with vertices corresponding to the parts of the regularity partition. We show there is a choice of red and blue edges in this cluster graph such that whenever there is a colored edge in the cluster graph, the pair of vertex sets in the regularity partition corresponding to that edge will have a large density in that color. We also show that with this choice, the resulting cluster graph is $\ep$-balanced and has density greater than $c$. We may then apply Lemma \ref{colored embedding lemma} to find our $t$-blowup. We now go through the details.

Let $G$ be an $n$ vertex $\ep$-balanced graph with $(c+\delta)\binom{n}{2}$ where $\ep$ is chosen as large as possible (so that the smaller color class has exactly $\ep(c+\delta)\binom{n}{2}$ edges). $\ep_0$ will be a positive constant depending on $\ep, \delta, \mathcal{F}$ that we can choose small enough for all of the subsequent calculations, and let $d: = \frac{\delta \ep^2}{4}$. Apply Corollary \ref{regularity corollary} to $G$ with parameter $\ep_0$ to acquire a regularity partition of $G$ with parts $\{R_i\}_{i=1}^k$. We apply the standard clean-up process:

\begin{itemize}
    \item Delete all edges incident on the junk set. There are at most $\ep_0n^2$ such edges.
    \item Delete all edges between pairs which are not $\ep_0$-regular. There are at most $\ep_0\binom{k}{2}(\frac{n}{k})^2\leq \frac{\ep_0n^2}{2}$ such edges.
    \item Delete all red edges between pairs with red density less than $d$. There are at most $\binom{k}{2}d(\frac{n}{k})^2\leq \frac{dn^2}{2}$ such edges.
    \item Repeat the previous step for blue edges. Again, there are at most $\frac{dn^2}{2}$ such edges.
    \item If necessary, delete edges from the appropriate color class in order to ensure the resulting graph is $\ep$-balanced.
\end{itemize}

Call the resulting graph $G'$. Note we deleted at most $(d+\frac{3\ep_0}{2})n^2$ edges before the last step. In the last step, we delete at most $\frac{1-\ep}{\ep}(d+\frac{3\ep_0}{2})n^2\leq \frac{\delta\ep}{2}\binom{n}{2}$ edges for a small enough choice of $\ep_0$.  Without loss of generality, call the smaller color class in the graph blue. We thus know that blue still has density at least $\ep(c+\frac{\delta}{2})$ and red still has density at least $(1-\ep)(c+\frac{\delta}{2})$ in the resulting graph $G'$. Denote these quantities $d_B(G')$ and $d_R(G')$.

Consider an auxiliary graph $\mathcal{R}$ whose vertices $P_i$ ($i\in[k]$) correspond to the $k$ parts in the regularity partition. We wish to add colored edges in $\mathcal{R}$ so that a red edge indicates high red density between the $\ep_0$-regular pair in $G'$, and similarly for blue. We wish to ensure this yields an $\ep$-balanced graph with density greater than $c$. To do so, we use Lemma \ref{lem:clusterlemma}. We can thus fix a coloring of the edges of $\mathcal{R}$ where the color of a pair indicates at least $d$ density of that color in the pair. Furthermore, we know that the density of red and blue in the coloring of $\mathcal{R}$ is at most $\gamma$ away from their respective densities in $G'$ (where $\gamma:=\max(\gamma_R, \gamma_B)$).
\par If necessary, delete red edges until the red density in $\mathcal{R}$ is at most $d_R(G')$. Further, if necessary, delete more edges from the appropriate color class until the resulting graph on $\mathcal{R}$ is $\ep$-balanced. As $d_B(G')=\frac{\ep}{1-\ep}d_R(G')$, we delete at most $\frac{(1-\ep)}{\ep}\gamma\binom{k}{2}$ edges in this step. Call the resulting graph $\mathcal{R}'$.

Now, note that we may choose $\ep_0$ small enough to control $\gamma$ so that $\mathcal{R}'$ has more than $(c+\delta/4)\binom{k}{2}$ edges. Since $\mathcal{R}'$ is $\ep$-balanced, we have that for $\ep_0$ small enough (so that $\mathcal{R}'$ has enough vertices), by definition of $c$, $\mathcal{R}'$ must contain a copy of some $F\in \mathcal{F}$. Since $\ep_0$ may also be chosen small enough with respect to $d$, we can apply Lemma \ref{colored embedding lemma} with the copy of $F$ in $\mathcal{R}'$ and conclude that $G'$ contains a $t$-blowup of $F$, completing the proof.
\end{proof}
\section{Non-monochromatic Cliques and Cycles}\label{nonmono section}

In this section, we will study the densities of $\frac{1}{2}$-balanced graphs that ensure the existence of either non-monochromatic cliques or non-monochromatic cycles, proving Theorems \ref{cycle prop} and \ref{clique prop}.
\subsection{Cliques}
Let $\mathcal{K}_k$ denote the family of non-monochromatic cliques on $k$ vertices. In this subsection we will show that 
    
    \[\mathrm{ex}\left(\frac{1}{2}, \mathcal{K}_k\right)\leq 1-\frac{1}{k+O(1)},\]
proving Theorem \ref{clique prop}
\begin{proof}

We modify a proof of Tur\'an's theorem attributed to Alon and Spencer \cite{aigner}.

Let $G$ be a $\frac{1}{2}$-balanced \bicolored graph with density $1-\frac{1}{k+C}:=\delta$, and choose $C$ large enough that $\delta > \frac{9}{10}$. Furthermore, assume that $G$ contains no non-monochromatic copy of $K_k$. We will show that $C$ must be bounded by an absolute constant (independent of $k$).

Consider the following algorithm to produce a non-monochromatic clique. First, sample the vertices with repetition until all vertices are seen, creating a sequence of vertices. Observe that this also induces a uniformly sampled permutation of the vertex set, and we will consider both the sequence with repetition and the corresponding permutation. Define $S$ to be the set of all vertices which appear before all of their non-neighbors in the permutation, and define $W$ to be the the number of vertices in the sequence (not the permutation) before a non-monochromatic triangle appears (note that $W$ will always be well-defined, as $\delta > \frac{2}{3}$ implies that there is at least $1$ non-monochromatic triangle). Let $T$ be the non-monochromatic triangle that appears first, and let $W'$ be the set of vertices that appear in the first $W$ terms of the sequence. 

By definition of $S$, the set $S$ induces a clique in the graph. Therefore, the set $T\cup S \setminus W'$ forms a non-monochromatic clique in $G$. Since $G$ is $\mathcal{K}_k$-free, this implies that 
\[
|T\cup S \setminus W'| < k.
\]

By linearity of expectation, we have that 
\[
\mathbb{E}[ |T \cup S \setminus W'| ] < k,
\]
and in particular $\mathbb{E}[|S|) - \mathbb{E}(W] \leq  \mathbb{E}[|S|) - \mathbb{E}(|W'|] < k.$

Now, we have that 
\[
\mathbb{E}[|S|] = \sum_{v\in V(G)} \frac{1}{n-d(v)}.
\]
By convexity of the function $f(x) = 1/x$, we have that 
\[
\mathbb{E}[|S|] \geq n\left(\frac{1}{n-\frac{1}{n}\sum d(v)}\right) = n \left(\frac{1}{ n - \delta(n-1)}\right) \sim \frac{1}{1-\delta} = k+C.
\]

Combining inequalities gives that $C < \mathbb{E}[W]$ and hence it suffices to show that $\mathbb{E}[W]$ is bounded by a constant. 

To do this, we claim that there are $\Omega(n^3)$ non-monochromatic triangles in $G$. To see this, in a graph on $n$ vertices with $e$ edges, there are at least $\frac{e(4e-n^2)}{3n}$ triangles (see \cite{MM} page 275). Since $\delta > \frac{9}{10}$, this implies that there are at least $\frac{12}{100}n^3$ triangles in $G$. On the other hand, in a graph with $\binom{x}{2}$ edges, there are at most $\binom{x}{3}$ triangles (see \cite{lovasz} Chapter 13, Exercise 31b). Since the red edges and the blue edges each individually have density less than $1/2$, this implies that there are at most $\frac{1}{6\sqrt{2}} n^3$ monochromatic triangles in $G$. Noting that $\frac{1}{6\sqrt{2}} < \frac{1}{12}$ proves the claim.

Now, to bound $\mathbb{E}[W]$ by a constant, we observe that in the sequence of vertices selected with repetition, every consecutive and disjoint triple is an independently and uniformly selected sequence of $3$ vertices. Since $G$ has $\Omega(n^3)$ non-monochromatic triangles, each of these sequences of $3$ vertices has a positive probability (independent of $k$) of inducing a non-monochromatic triangle, and hence the expected waiting time to see a non-monochromatic triangle is $O(1)$. 

  \end{proof}
\subsection{Cycles}
Here, our goal is to prove Theorem \ref{cycle prop}, establishing that with the exception of triangles, the extremal threshold for finding non-monochromatic cycle of any length in $(1/2)$-balanced graphs is $1/4$. We will end up proving a result significantly stronger, by finding the extremal threshold of all cycles where one of the color classes is a disjoint union of even length paths. (There are cycles with inevitable colorings which are not of this form, but the smallest example is on $8$ vertices.)
\begin{theorem}\label{cycle theorem} Let $C$ be an inevitable cycle such that all maximal (without loss of generality) blue paths are of even edge length. Then we have the extremal value dichotomy:
\begin{enumerate}
    \item If the red (or blue) color class in $C$ contains any isolated edges, then $\mathrm{ex}(\frac{1}{2}, n, C)= \frac{2}{3}$
    \item Otherwise, all maximal red (or blue) paths are of length at least $2$, and $\mathrm{ex}(\frac{1}{2}, n, C)= \frac{1}{2}$
\end{enumerate}
\end{theorem}
Theorem \ref{cycle prop} follows from the second case of the dichotomy, as there exist non-monochromatic inevitable colorings of cycles (of length greater than $3$) with the blue color class only containing even length maximal paths and no isolated edges in red.
\par The theorem largely depends on the next lemma. As before, let $T_1$ and $T_2$ be the two types of non-monochromatic triangles, and let $H_1$ and $H_2$ be a red triangle incident with a blue edge and a blue triangle incident with a red edge respectively, and call these \bicolored graphs {\em handles}.
\begin{lemma} $\mathrm{ex}(\frac{1}{2}, \{T_1, T_2, H_1, H_2\})= \frac{1}{2}$
\end{lemma} \label{handle lemma}
We delay the proof of the lemma to Section \ref{handle section}, and we first show how it implies the main theorem of this section. Our technique is to use our Erd\H{o}s-Stone type theorem (Theorem \ref{ES theorem}) to find blow-ups of either non-monochromatic triangles or handles and finding appropriate embeddings inside the blow-ups.
\begin{proof}[Proof of Theorem \ref{cycle theorem}] There are two cases, depending on whether the cycle $C$ has isolated edges. 
\par \textbf{Case 1: } If the cycle $C$ contains isolated edges, we will use same construction from DeVos, McDonald, and Montejano (\cite{DMM}) which is as follows. Split the vertex set into three equal sized parts, add red edges between a pair of parts and add blue edges for another pair of parts. Let the part incident on both red and blue edges be independent. Let the other pair of the red bipartite graph be a red clique, and the other pair of the blue bipartite graph a blue clique. Add no edges between the red and the blue clique. It is easy to see that this construction contains no cycles with an isolated edge in either color class. And the density of this construction is $2/3$ whenever $n$ is a multiple of $3$. This establishes the lower bound. 
\par Now, we consider a large enough $(1/2)$-balanced graph with density $2/3 +\delta$ for some fixed positive $\delta$. In \cite{DMM} it is shown that $\mathrm{ex}(\frac{1}{2}, \mathcal{K}_3)=2/3$. If the graph is large enough, by Theorem \ref{ES theorem} we can find a $v(C)$-blow-up of a non-monochromatic triangle. To conclude this case, we just have to find an embedding of $C$ (whose blue edges are a disjoint union of even length paths) into this blow-up. 
\par To achieve this, we induct on the number of red and blue edges that are adjacent in $C$. Note that this has to be an even number. By the inductive hypothesis, we only have to embed a even length blue path followed by a red path that starts and ends on the same part of the blow-up. This is easy to do, casing on whether the red path is of even or odd length. 
\par \textbf{Case 2: } In this case, the blue paths are still of even length, and red paths are of length at least two. As $1/2$ is the trivial lower bound by the disjoint red and blue clique construction, we only need to show the upper bound. So assume we have a large enough $1/2$ balanced graph whose density exceeds $1/2$.
\par By our Lemma \ref{handle lemma} and Theorem $\ref{ES theorem}$ we can find a large enough blow-up of either a non-monochromatic clique or a handle. To complete the embedding, just like in the previous case, by induction we are only concerned with an addition of a blue path of even length, and a red path of length at least $2$, and we need to embed a path that starts and ends on the same part of the blow-up. If we have to embed to a blow-up of a non-monochromatic clique, we embed the red part of the graph to the bipartite graph whose color is different from the other two. To embed to a blow-up of a handle, we need to make sure to embed the blue part on the bipartite graph whose color is different from that of the monochromatic clique. Both cases are straightforward and we omit further details.
\end{proof}

\subsection{Handles and nonmonochromatic triangles}\label{handle section}
Our goal is now reduced to proving Lemma \ref{handle lemma}, i.e. finding either handles or non-monochromatic triangles in $1/2$-balanced bicolored graphs with edge density $1/2$. Note the lower bound in the lemma simply follows from coloring the edges of a $K_{n/2,n/2}$ half red, half blue, in an arbitrary fashion. 

Throughout this section, we fix $G$ to be a sufficiently large $1/2$-balanced graph containing neither a handle nor a nonmonochromatic triangle. We are done if we can show that $e(G) \lesssim n^2/4$. We fix a partition of the vertices of $G$: $$V(G)=R\sqcup B \sqcup M.$$ 

\par Here, $R$ (resp. $B$) denotes the vertices of $G$ that are only incident on red (resp. blue) edges. $M$ denotes the vertices that are incident on at least one red and one blue edge. Any vertex not in one of these sets would be an isolated vertex, the deletion of which would create a denser graph. Hence, we may assume $R$, $B$, and $M$ partition $V(G)$. Let $r$, $b$, and $m$ denote the sizes $R$, $B$ and $M$ respectively. 

\begin{lemma}\label{M triangle free}
For every triangle $T$ in $G$, $T\cap M=\emptyset$. In particular, $e(G[M])\leq m^2/4$, and for any edge $\{u,v\}$ contained in $R$ or $B$, $N(u)\cap N(v)\cap M =\emptyset$. 
\end{lemma}
\begin{proof}
Assume for the sake of contradiction there exists a triangle $T$ which has a vertex from $M$, call it $x$. We may assume the triangle is monochromatic, as $G$ is handle and nonmonochromatic triangle free. As $x$ is incident on both a red and a blue edge, regardless of the color of $T$, we may extend it to an handle through $x$. The bound on $e(G[M])\leq m^2/4$ immediately follows by Mantel's theorem. Adjacent vertices in $R$ or $B$ cannot have a common neighborhood in $M$, as this gives a triangle with a vertex inside $M$. 
\end{proof}

\begin{lemma}\label{M big independent set}
If a triangle free graph $H$ on $m$ vertices has an independent set of size $d$ where $d\geq m/2$, then $H$ has at most $d(m-d)$ edges.
\end{lemma}
\begin{proof}
If $d\geq m/2$ then the quantity $d(m-d)$ is decreasing in $d$, so without loss of generality we may assume that the independence number of $H$ is $d$. Let $I$ be an independent set of size $d$. Since $H$ is triangle-free, the neighborhood of any vertex is an independent set, and so $d(v) \leq d$ for all $v$. Since $I$ is an independent set, any edge must have at least one vertex not in $I$. Therefore 
$$e(H) \leq \sum_{v\not \in I} d(v) \leq \sum_{v\not\in I} d = (m-d)\cdot d.
$$ 
\end{proof}

Our approach is to modify $G$ into $G'$, decreasing neither the red nor the blue edges, until $G'$ has an easy to understand structure. Using the bound on $e(G[M])$ from Lemmas \ref{M triangle free} and \ref{M big independent set}, we will be able to bound the density of $G'$ from above.

\par For $v\in G$, let $G\cup \text{clone}(v)$ denote the graph obtained by adding to $G$ a new vertex $v'$ with $N(v')=N(v)$, and the color of an edge $\{v', x\}$ is same as that of $\{v, x\}$. In particular, $v$ and $v'$ are not adjacent in $G\cup \text{clone}(v)$. 

\begin{lemma}\label{clones} 
For any \bicolored $H$ avoiding handles and nonmonochromatic triangles, and $v\in V(H)$, $H\cup \text{clone}(v)$ still avoids handles and nonmonochromatic triangles. 
\end{lemma}
\begin{proof}
It is obvious that any new forbidden substructure in $H\cup \text{clone}(v)$ has to use both $v$ and the clone, $v'$. Since $v$ and $v'$ are non-adjacent, this structure can only be a handle. Yet, for any monochromatic triangles $v$ takes part in, $v'$ is adjacent to all three vertices of this triangle in the same color. Hence there cannot be handles in the new graph either. 
\end{proof}

Fix a permutation of $V(G)$ arbitrarily and set $G_0:=G$. While there is still a non-edge $\{u,v\}$ contained in either $R$ or $B$ such that $N(u)\neq N(v)$ in $G_i$ we perform the following operation to obtain $G_{i+1}$: we take $x\in\{u,v\}$ with larger degree (settle ties by index in the permutation), delete the lower degree vertex to acquire $G_i'$, and set  $G_{i+1}:=G_i'\cup \text{clone}(x)$. When the procedure terminates, we set the final graph to be $G'$. 
\par By Lemma \ref{clones}, $G'$ does not contain any forbidden substructure, and since we perform the cloning within $R$ and $B$, neither the red nor the blue edge count decreases at any point during the procedure. 
\begin{observation}
In $G'$, non-adjacency is an equivalence relation in $R$ and $B$. Thus, both $G'[R]$ and $G'[B]$ are complete multipartite graphs. 
\end{observation}
Indeed, if $\{x,y\}$ and $\{y,z\}$ are non-edges in $R$ (or $B$), by virtue of the procedure, $N(x)=N(y)=N(z)$, meaning that there cannot be an edge between $x$ and $z$. Since vertices in each maximal independent set of $R$ (and $B$) are clones of each other, their neighborhoods in $M$ are also identical, i.e. each equivalence class of non-edges in $R$ (and $B$) induces a complete bipartite graph between $R$ (or $B$) and $M$. Call $R_1,\cdots, R_k$ be the parts of $R$, and $B_1,\cdots B_l $ the parts of $B$. The next observation follows from Lemma \ref{M triangle free}.

\begin{observation}
For any $i\not=j$, the sets $N(B_i)\cap N(B_j)\cap M$ and $N(R_i)\cap N(R_j)\cap M$ are empty. 
\end{observation}

This observation implies that every vertex in $M$ sends edges to at most one $R_i$ and at most one $B_j$. Using this, we may refine the structure even more. Without loss of generality, assume that $|R_1|  \geq \cdots \geq |R_k|$ and $|B_1| \geq \cdots \geq |B_l|$. Define
\[
d:= \max \{|N(R_i) \cap M|, |N(B_j) \cap M|\}_{i,j}.
\]

By Lemma \ref{M triangle free}, $M$ contains an independent set of size $d$. Since $R$ and $B$ induce complete multipartite graphs, we have 
\begin{equation}\label{edges in R}
    e(R) = \binom{r}{2} - \sum_{i=1}^k \binom{|R_i|}{2}  \leq \binom{r}{2} - \binom{|R_1|}{2} - \binom{|R_2|}{2},
    \end{equation}
    \begin{equation}\label{edges in B}
    e(B) = \binom{b}{2} - \sum_{i=1}^l \binom{|B_i|}{2}  \leq \binom{b}{2} - \binom{|B_1|}{2} - \binom{|B_2|}{2}.
\end{equation}

Because each vertex in $M$ sends edges to at most one independent set in $R$ and at most one independent set in $B$, and because $|R_1| \geq \cdots \geq |R_k|$ and $|B_1| \geq \cdots \geq |B_l|$, we have
\begin{equation}\label{edges between M and R and B}
e(R,M) + e(B, M) \leq d(|R_1| + |B_1|) + (m-d)(|R_2| + |B_2|).
\end{equation}

Consider the following graph $G''$. Let $V(G'') = R'\sqcup M' \sqcup B'$, where $|R'| = r$, $|B'| = b$, and $|M'| = m$. Let $M' = M_1 \sqcup M_2$ where $|M_1| = \max \{m/2, d\}$. We will define the edges of $G''$ in $R'$ and $B'$, those with one endpoint in $M$ and then those with both endpoints in $M'$. 

Let $R'$ induce a complete multipartite graph with parts of size $|R_1|$, $|R_2|$, and all the rest of size $1$. Let $B'$ induce a complete multipartite graph with partite sets of size $|B_1|$, $|B_2|$, and all the rest of size $1$.

Place a complete bipartite graph between the independent set of size $|R_1|$ and $M_1$, between the independent set of size $|R_2|$ and $M_2$, between the independent set of size $|B_1|$ and $M_1$, and between the independent set of size $|B_2|$ and $M_2$.

Let $M_1$ and $M_2$ induce a complete bipartite graph in $M'$. 
\begin{figure}[t]
    \includegraphics[width=\textwidth]{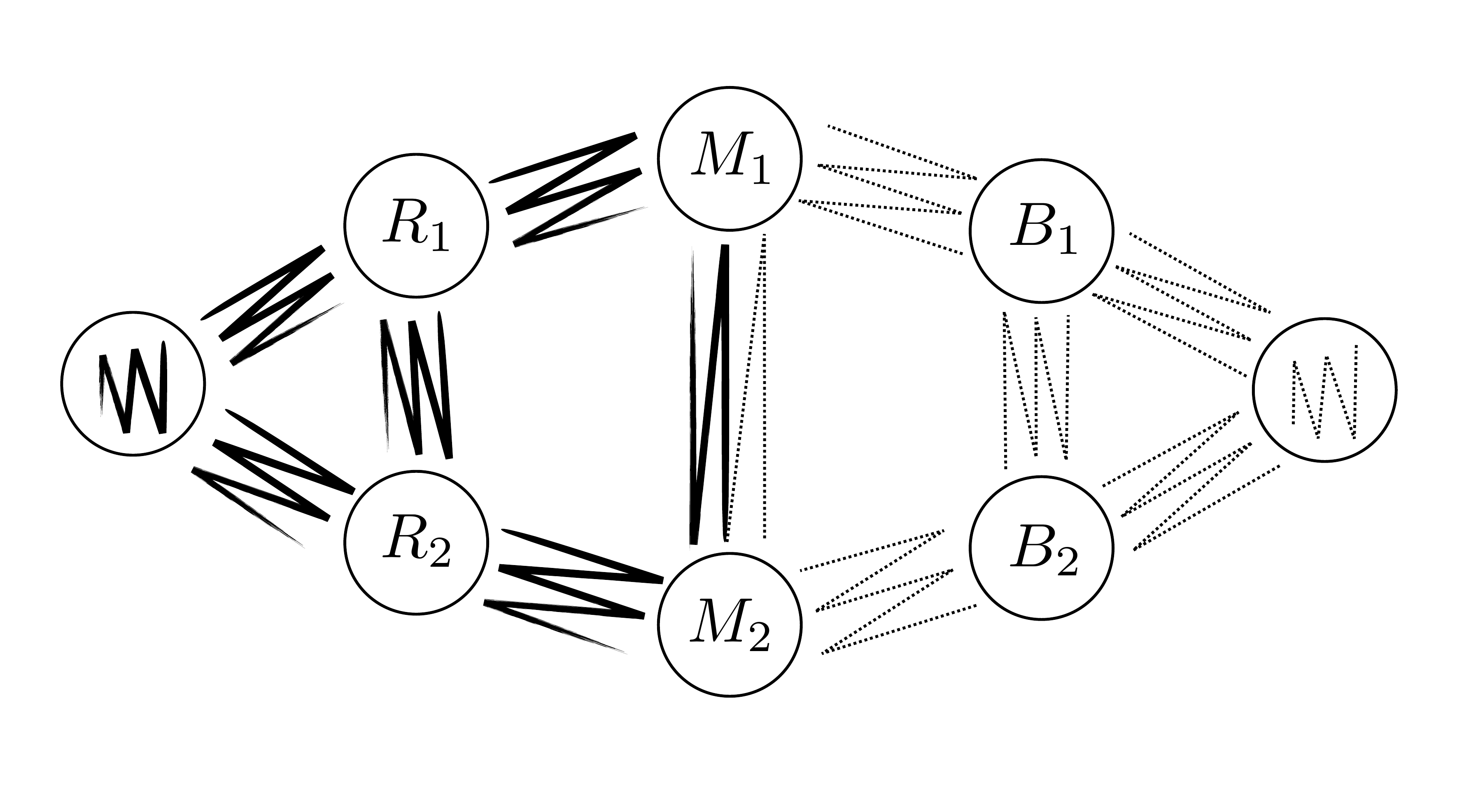}
    \caption{Graph $G''$. Solid indicates red, dotted indicates blue color. There are no restrictions on color of edges between $M_1$ and $M_2$, they can be red or blue. All labelled parts are independent sets, and the leftmost and rightmost sets are red and blue cliques respectively. The leftmost three sets union to $R'$, and rightmost three sets union to $B'$. 
}
    \label{fig:my_label}
\end{figure}

By Lemma \ref{M triangle free}, we have that $e(M') \leq m^2/4$. If $d\geq m/2$, by Lemma \ref{M big independent set} we have that $e(M) \leq d(m-d)$. Therefore, by equations \eqref{edges in R}, \eqref{edges in B}, and \eqref{edges between M and R and B}, we have $e(G'') \geq e(G')$, and so it suffices to show that $e(G'') \lesssim n^2/4$.

We claim that $|R'|$ and $|B'|$ must each be bounded above by $n/2$. Otherwise, say $|R'| >n/2$, we would have $|M\cup B| <n/2$ and so that number of blue edges in $G$ would be less than $\binom{n/2}{2}$. We will show that $e(G'') \lesssim n^2/4$ for any possible sizes of the parts, subject to the constraint that $|R'|, |B'| \leq n/2$. Assume that $G''$ has the maximum number of edges possible. Since we are only interested in an asymptotic result, we will work with densities to ease the calculations. Let $r' = |R'|/n=r/n$, $b' = |B'|/n = b/n$, $r_i = |R_i|/n$, $b_i = |B_i|/n$ and $m_i = |M_i|/n$. We will suppress negligible error terms, so for example if $|R'| \sim n/4$ we will write that $r'=\frac{1}{4}$ instead of $r' = \frac{1}{4} + o(1)$. Note that as long as $|R'|\not=|R_1|$ and $|B'|\not=|B_1|$, the sets $R_i$ and $B_i$ will be nonempty, since they are the sizes of the largest independent sets in $R$ and $B$ in $G$.

\textbf{Case 1: $b'=0$ or $r'=0$:} Without loss of generality, assume $b'=0$. If we have $r'-r_1-r_2=0$ and $b'=0$, then $G''$ is $o(n^2)$ edges from being bipartite, and hence we have $e(G'') \lesssim n^2/4$. Otherwise $R' \setminus (R_1 \cup R_2)$ is nonempty. Moving a vertex from $R'\setminus (R_1\cup R_2)$ to $R_i$ changes the number of edges in $G''$ by $(m_i - r_i)n$. Since $e(G'')$ is maximized, we must have $m_1=r_1$ and $m_2=r_2$. But since $b'=0$ and $r'\leq 1/2$, we have $m_1 + m_2 \geq 1/2$ and so $r_1+r_2$ must equal $1/2$. But now again $G''$ is $o(n^2)$ edges from being bipartite and we are done. 

\textbf{Case 2: $r'>0$, $b' >0$:} If $r' - r_1 -r_2 = 0$, then relabeling vertices in $R_1$ to be in $M_2$ and vertices in $R_2$ to be in $M_1$ changes the graph by $o(n^2)$ edges, and reduces to Case 1. Similarly, if $b'-b_1-b_2=0$, we may reduce to Case 1. So we may assume that both $r' - r_1-r_2$ and $b'-b_1-b_2$ are positive. Without loss of generality, assume that $r'\geq b'$.

Now, moving a vertex from $R_1$ to $R_2$ (recall that both are nonempty) changes the number of edges by $(r_1 - r_2+ m_2-m_1)n$ and moving a vertex from $B_1$ to $B_2$ changes the number of edges by $(b_1-b_2 + m_2-m_1)n$. Hence we have 
\[
r_1 - r_2 = m_1 - m_2 = b_1 - b_2.
\]

If $b' = 1/2$, then $r'=1/2$ (since we assumed $r'\geq b'$) and $m=0$, and $e(G'') \lesssim 2\binom{n/2}{2}$. So assume that $b' < 1/2$, and in particular $B'$ could gain vertices without violating the constraint. Then moving a vertex from $M_2$ to $B_1$ changes $e(G'')$ by $((b'-b_1-b_2) - r_2)n$ and moving a vertex from $M_1$ to $B_2$ changes $e(G'')$ by $((b'-b_1-b_2) - r_1)n$. Thus we have 
\[
r_1 = r_2 = b'-b_1-b_2,
\]
which implies that $m_1 = m_2$ and $b_1=b_2$. Now, moving a vertex from $B'\setminus (B_1\cup B_2)$ to $B_i$ changes $e(G'')$ by $(m_i - b_i)n$ and moving from $R'\setminus (R_1\cup R_2)$ to $R_i$ changes $e(G'')$ be $(m_i -r_i)n$. Therefore 
\[
r_1=r_2=b_1 = b_2 = m_1 = m_2 =b'-b_1-b_2.
\]

Finally, moving a vertex from $B'\setminus (B_1\cup B_2)$ to $R' \setminus (R_1 \cup R_2)$ (if it is possible without violating the constraint), changes the number of edges by $(r-b)n$. We therefore must have that either $r'=b'$ or $r=1/2$. 

When $r'=b' <1/2$, moving a vertex from $M_2$ to $R_1$ or $M_1$ to $R_2$ shows that $b_1 =b_2 = r'-r_1-r_2$. In this case we have that all $8$ parts have the same size, and $e(G'') \sim \frac{3}{16}n^2$. 

Otherwise, $r'=1/2$ and so $m_1+m_2 + b' = 1/2$. It follows that the seven parts which are the same size all have size $\frac{n}{10}$, and in this case $e(G'') \sim \frac{4}{25}n^2$.

In all cases, we have that $e(G) \leq e(G'') \lesssim n^2/4$, completing the proof of Lemma \ref{handle lemma}.

\section{Discussion}
Our Theorem \ref{thm:mainextremal} gives an upper bound on the $\ep$-balanced Tur\'an number of a family as a function of the $\ep$-balanced Ramsey number of a complete graph, and blowups of $\ep$-balanced Ramsey graphs can also give lower bounds on the $\ep$-balanced Tur\'an number of a family. In general, we do not have any methods for determining exact values of either the $\ep$-balanced Ramsey function or the $\ep$-balanced Tur\'an function, and it would be very interesting to understand when the extremal graphs for the $\ep$-balanced Tur\'an function are given by blowups of $\ep$-balanced Ramsey graphs. In particular, we believe this to be true asymptotically when avoiding a non-monochromatic $K_k$ with $k$ sufficiently large (Conjecture \ref{nonmono clique conjecture}).

One potential difficulty in proving Conjecture \ref{nonmono clique conjecture} is that the equality does not hold for $k\in \{3, 4\}$. The theorem of DeVos, McDonald, and Montejano \cite{DMM}, that $\mathrm{ex}(\frac{1}{2}, \mathcal{K}_3)=2/3$, shows that a $\frac{1}{2}$-balanced complete bipartite graph is not extremal for avoiding a non-monochromatic triangle. For $k=4$, consider the following construction due to Urschel \cite{john}. Let $A$ be a clique and let $B,C,D$ be the partite sets of a complete $3$-partite graph. Join $A$ to the graph induced by $C\cup D$. Any $K_4$ in this graph must use at least one vertex from $A$. Therefore, to avoid a nonmonochromatic $K_4$, color blue all edges within $A$, between $A$ and $C\cup D$, and between $C$ and $D$. Color the edges between $B$ and $C\cup D$ so that the whole graph is $\frac{1}{2}$ balanced (note, this puts constraints on the sizes of $A$, $B$, $C$, $D$). Optimizing over the sizes of the parts gives that $\mathrm{ex}(\frac{1}{2}, \mathcal{K}_4) > 0.67508 > \frac{2}{3}$.  It would be interesting to determine the exact value of $\mathrm{ex}(\frac{1}{2},\mathcal{K}_4)$.

In a more general setting, removing the constant factor in our Theorem \ref{thm:mainextremal} would be quite interesting, as already stated in the Introduction. Denoting the family of $t$-unavoidable graphs by $\mathcal{I}_t$, Theorem \ref{thm:mainextremal} says $\textup{ex}(\ep, \mathcal{I}_t)\leq  1-\frac{1}{O(R(\ep',t))}$, where the suppressed constant term depends on how close $\ep'$ is to $\ep$. On the other hand, taking a blow-up of $\ep$-balanced Ramsey graphs shows that $\mathrm{ex}(\ep, \mathcal{I}_t)\geq 1-  \frac{1}{R(\ep, t) -1}$. 
\par Removing the constant term that appears in the upper bound for $\mathrm{ex}(\ep, \mathcal{I}_t)$ would show that blow-up of Ramsey graphs are asymptotically the densest construction avoiding unavoidable graphs. One way of achieving this would be through removing the constant factor $C$ in Lemma \ref{lem:mainlemma}. Explicitly, it would be interesting to determine if for any $\ep'<\ep$, any sufficiently large graph with density $1-\frac{1}{k}$ contains an $\ep'$-balanced graph on $(1-o(1))k$ vertices.

We mention a conjecture of a similar flavor by Diwan and Mubayi \cite{DM}. They conjecture that given any red-blue edge coloring of a $K_{k}$, then any union of  a red and blue graph each with at least $\mathrm{ex}(n, K_k)$ edges (the union will be a multigraph) will contain this colored $K_k$. They prove their conjecture when the coloring contains a monochromatic clique on $k-1$ vertices. They also ask whether an Erd\H{o}s-Simonovits-Stone theorem is possible in this context. Our Theorem \ref{ES theorem} can give some information in this setting, but we were not able to define a chromatic number-like parameter that would satisfactorily answer their question. 

It would be interesting to investigate $\mathrm{ex}(\ep, \mathcal{K}_k)$ and $\mathrm{ex}(\ep, \mathcal{C}_k)$ when $\ep < \frac{1}{2}$.

Finally, in this paper we studied what happens when one forbids a color-consistent family of graphs. It would be interesting to understand what happens when forbidding, for example, a nontrivial subset of the family of non-monochromatic complete graphs or cycles.

\section*{Acknowledgements}
We would like to thank Boris Bukh for many helpful discussions and ideas throughout the project, Bernard Lidick\'y for doing flag algebra calculations for us, and John Urschel for the construction of graphs avoiding nonmonochromatic $K_4$.

\bibliographystyle{plain}
\bibliography{bib}
\newpage
\section{Appendix}
Here we provide the details for the proof of Theorem \ref{thm:linearramsey} which we omitted in the paper, as it closely resembles the original proof of Theorem \ref{chvatal}.
\begin{proof}[Proof of Theorem \ref{thm:linearramsey}]
Let $H$ be an inevitable graph with maximum degree $\Delta$. Given $\ep$, and $H$, we will fix a regularity parameter $\ep_0$, whose value is chosen later. We also fix an $\ep$-balanced graph $G$ with $v(G)\geq C\cdot v(H)$, where $C$ is a constant that depends on $\ep_0$ and $\Delta$, exact value to be specified later. Our goal is to find a copy of $H$ in $G$. 
\par We apply Corollary \ref{regularity corollary} to $G$ with parameter $\ep_0$. We delete all edges incident to junk set (at most $\ep_0n^2$) and between pairs which are not $\ep_0$-regular (at most $\ep_0n^2/2$), and all red edges between parts with red density less than $\ep/10$ (at most $\ep n^2/20$), and similarly for blue. In total, we have lost at most (being generous) $(\ep/5 + 4\ep_0)\binom{n}{2}$ edges. In particular, choosing $\ep_0$ small will ensure that the graph is still at least $(\ep/2)$-balanced. Call this resulting graph $G'$. Note that $G'$ has colored edges between every $\ep_0$-regular pair, and for every such pair, at least one of the color classes will have density greater than $\ep/10$. 
\par We use Lemma \ref{lem:clusterlemma} with the auxiliary graph $\mathcal{R}$ whose vertices $P_i$ ($i\in[k]$) correspond to the $k$ parts in the cluster graph to find a coloring of $\mathcal{R}$ where the color of a pair indicates at least $(\ep/10)$ density of that color in the pair. And further, the density of the smaller color class changed at most by $\gamma:=\max(\gamma_R, \gamma_B)$. Since we choose $\ep_0$ after $\ep$ is fixed, we can choose $\ep_0$ small enough to bound $\gamma< \ep/100$. Thus, since $G'$ was $\ep/2$ balanced, the coloring of $\mathcal{R}$ will surely be at least $(\ep/4)$-balanced. And furthermore, we know that the density of this \bicolored $\mathcal{R}$ is at least $1-\ep_0$. 
\par Thus, we can apply our Theorem \ref{thm:mainextremal} to say there exists a constant $K$ such that for any $t$ which satisfies
$$\ep_0 \leq \frac{1}{K\cdot R(\ep/4, t)}$$
we can find an unavoidable $t$-graph in the color cluster graph. Using bounds on $R(\ep/4,t)$, we obtain that for a constant $K'$, we can set $t\geq K'\log(\ep_0^{-1})$. Note this allows us to make $t$ as large as we want by making $\ep_0$ small. Let's choose $\ep_0$ so that $K'\ep_0^{-1}\geq \Delta + 1$. 
\par We are at the last section of the proof where we just have to embed $H$ into the clusters. Let's say that the unavoidable graph in the cluster graph we found was of Type $1$ (the proof is similar if it is of Type $2$). As $H$ is inevitable, $H$ has an embedding to a large enough Type $1$ graph. In both the left hand side and right hand side of this embedding, the maximum degree is at most $\Delta$, hence there exists a partition of each side into $\Delta + 1$ independent sets. Now, we will go through the independent sets and embed each vertex in a greedy fashion to the corresponding cluster, using only $\ep_0$-regularity, and that the density between each pair is at least $\ep/10$.

\par Our goal is to find an embedding $f:V(H) \to V(G)$ which respects incidence. We will find the embedding iteratively and greedily. For each vertex $u$ we will maintain a set of ``potential" vertices in $G$ such that if $u$ were mapped to any of these vertices, incidence would be respected. We do this as follows.

\par  For any vertex $u\in V(H)$, denote by $\mathcal{C}_{u, i}$ the set of potential vertices in the corresponding cluster to which $u$ can be embedded after $i$ vertices have already been embedded. More precisely, $\mathcal{C}_{u, i}$ is composed of vertices $v\in \mathcal{C}_{u, i-1}$ such that for all $x\in V(H)$ whose embedding to the graph $f(x)\in V(G)$ is defined by the $i^{th}$ step $u\sim x \implies v\sim f(x)$. This condition ensures that the embedding preserves incidence. We will prove by induction that $\mathcal{C}_{u,i}$ is always larger than $v(H)$, and thus we may always choose an $f(u)$ such that $f$ is both an injection and respects incidence. We set $\mathcal{C}_{u, 0}$ to be just the entirety of the cluster we plan to embed $u$ inside. Say we are in the $(i+1)^{th}$ stage of the algorithm. Let $g(u,i):=\# \text{ of neighbors of $u$ already embedded before step $i$}$.  Assume inductively that
$$|\mathcal{C}_{u, i}|\geq |\mathcal{C}_{u, 0}|\cdot (\ep/10 - \ep_0)^{g(u,i)} $$
for all unembedded vertices $u$. So for some fixed unembedded $u$, we have a set $\mathcal{C}$ with  $|\mathcal{C}|\geq |\mathcal{C}_{u,0}|(\ep/10 - \ep_0)^{g(u,i)} - v(H)$ that we can choose for $f(u)$, accounting for vertices from the same cluster that might already have been used. To select $f(u)$ in a way that will preserve the induction hypothesis, we invoke $\ep_0$-regularity between $\mathcal{C}$ and all the $\mathcal{C}_{u', i}$ for $u'\sim u$ and $u'$ is not embedded yet. We may do so, as long as $|\mathcal{C}_{u,0}|\ep_0<|\mathcal{C}_{u,0}|(\ep/10 - \ep_0)^{\Delta} - v(H)$. Since $|\mathcal{C}_{u,0}| = v(G) / k \geq C\cdot v(H)/k$ where $k$ is the number of clusters, we may choose $\ep_0$ small enough, and then $C$ large enough (the number of clusters is bounded above depending only on $\ep_0$) so that the inequality holds. Thus, we can apply $\ep_0$ regularity which tells us that all but at most $\Delta \ep_0$ fraction of $\mathcal{C}$ has at least $(\ep/10 - \ep_0)|\mathcal{C}_{u',i}|$ neighbors in each $\mathcal{C}_{u',i}$ for each unembedded $u'$. As $(1-\Delta\ep_0)|\mathcal{C}|\geq 1$, we can choose a $f(u)$ that preserves our inductive invariant. Hence we will be able to embed all of $H$ in $G$. \end{proof}

\end{document}